\documentclass[a4paper]{amsart}
\usepackage[english]{babel}
\usepackage[utf8]{inputenc}
\usepackage[T1]{fontenc}
\usepackage{lmodern}

\usepackage[shortlabels]{enumitem}
\setlist[enumerate]{label={\arabic*.}}
\setlist[description]{font=\normalfont\itshape}

\usepackage[dvipsnames]{xcolor}
\usepackage[backref]{hyperref}
\hypersetup{colorlinks=true,linkcolor=Maroon,citecolor=OliveGreen}

\usepackage{cleveref}
\usepackage{microtype}

\usepackage{setspace}
\usepackage{amsmath,amsfonts,amsthm,mathtools,commath,amssymb}

\def\N{\mathbf{N}}
\def\Z{\mathbf{Z}}

\def\R{\mathbf{R}}
\def\C{\mathbf{C}}
\def\F{\mathbf{F}}
\def\P{\mathbf{P}}
\def\Pr{\mathbf{P}}

\DeclarePairedDelimiter\floor{\lfloor}{\rfloor}

\renewcommand{\gcd}{\mathrm{gcd}}

\DeclareMathOperator{\ad}{ad}
\DeclareMathOperator{\tr}{tr}
\DeclareMathOperator{\lin}{Lin}
\DeclareMathOperator{\lspace}{L^2}
\DeclareMathOperator{\slfrak}{\mathfrak{sl}}
\DeclareMathOperator{\gfrak}{\mathfrak{g}}

\DeclareMathOperator{\SL}{SL}
\DeclareMathOperator{\Sym}{Sym}
\DeclareMathOperator{\real}{Re}
\DeclareMathOperator{\dist}{d}
\DeclareMathOperator{\Stab}{Stab}
\DeclareMathOperator{\id}{id}
\DeclareMathOperator{\rank}{rank}
\DeclareMathOperator{\minnorm}{\Delta}

\DeclareMathOperator{\diam}{diam}

\newtheorem{theorem}{Theorem}[section]
\newtheorem*{theorem*}{Theorem}
\newtheorem{lemma}[theorem]{Lemma}
\newtheorem{proposition}[theorem]{Proposition}
\newtheorem{corollary}[theorem]{Corollary}
\newtheorem*{example*}{Example}

\theoremstyle{definition}

\newtheorem{remark}[theorem]{Remark}

\newtheorem*{conjecture*}{Conjecture}

\usepackage{todonotes}

\begin{document}
\baselineskip=13pt % comfortable reading

\title{Random Lie bracket on $\mathfrak{sl}_2(\F_p)$}

\author{Urban Jezernik}
\address{Urban Jezernik, Faculty of Mathematics and Physics, University of Ljubljana and Institute of Mathematics, Physics, and Mechanics, Jadranska 19, 1000 Ljubljana, Slovenia}
\email{urban.jezernik@fmf.uni-lj.si}

\author{Matevž Miščič}
\address{Matevž Miščič, Faculty of Mathematics and Physics, University of Ljubljana and Institute of Mathematics, Physics, and Mechanics, Jadranska 19, 1000 Ljubljana, Slovenia}
\email{matevz.miscic@imfm.si}

\thanks{UJ is supported by the Slovenian Research and Innovation Agency program P1-0222 and grants J1-50001, J1-4351, J1-3004, N1-0217. MM is supported by the Slovenian Research and Innovation Agency program P1-0222 and by the Institute of Mathematics, Physics, and Mechanics.}
\begin{abstract}
    We study a random walk on the Lie algebra $\slfrak_2(\F_p)$ where new elements are produced by randomly applying adjoint operators of two generators. Focusing on the generic case where the generators are selected at random, we analyze the limiting distribution of the random walk and the speed at which it converges to this distribution. These questions reduce to the study of a random walk on a cyclic group. We show that, with high probability, the walk exhibits a pre-cutoff phenomenon after roughly $p$ steps. Notably, the limiting distribution need not be uniform and it depends on the prime divisors of $p-1$. Furthermore, we prove that by incorporating a simple random twist into the walk, we can embed a well-known affine random walk on $\F_p$ into the modified random Lie bracket, allowing us to show that the entire Lie algebra is covered in roughly $\log p$ steps in the generic case.
\end{abstract}

\maketitle

\section{Introduction}

Random walks on finite groups are a powerful tool for analyzing their structure and properties (see \cite{diaconis1988group, hildebrand2005survey} and references therein). Of particular significance is their connection to growth and expansion in Cayley graphs of finite simple groups (see, for example, \cite{bourgain2008expansion}). However, extending these techniques, and ultimately the results, to the closely related setting of finite simple Lie algebras is not straightforward. Unlike groups, Lie algebras have two operations: addition and the Lie bracket. The fact that the Lie bracket is non-invertible introduces additional challenges.

In this paper, we introduce and study a random walk on the Lie algebra $\slfrak_2(\F_p)$ that focuses solely on the Lie bracket structure. This walk is defined as follows. Let $A$ and $B$ be a generating pair of the Lie algebra $\slfrak_2(\F_p)$ with $p \geq 3$. Let $Z_k$, for $k \in \N$, be independent random variables uniformly distributed in the set $\{ A, B \}$. Starting with $X_0 = [A, B]$, the random walk evolves according to
\[
    X_k = [Z_k, X_{k - 1}]
    \qquad (k \in \N).
\]
This process generates a sequence $(X_k)_{k \in \N_0}$ that forms a Markov chain on the set $\slfrak_2(\F_p)$. New elements are produced by iteratively applying the adjoint operators associated with the generators, a process we refer to as the \emph{random Lie bracket}.

\begin{example*}
    Let us look at a simulation of the random Lie bracket. Suppose 
    \[
    p = 11, \quad
    A = \begin{pmatrix}
        7 & 5 \\ 2 & 4
    \end{pmatrix}, \quad
    B = \begin{pmatrix}
        8 & 8 \\ 10 & 3
    \end{pmatrix}, \quad
    X_0 = [A,B] = 
    \begin{pmatrix}
     1 & 10 \\
     2 & 10 \\
    \end{pmatrix}.
    \]
    Randomly sample the sequence $(B,A,B,B,B)$ from $\{ A, B \}$ to obtain the values
    \[
        X_1 = \begin{pmatrix}
            4 & 1 \\
            10 & 7 \\
        \end{pmatrix}, \
        X_2 = \begin{pmatrix}
            4 & 7 \\
            8 & 7 \\
        \end{pmatrix}, \
        X_3 = \begin{pmatrix}
            5 & 4 \\
            7 & 6 \\
        \end{pmatrix}, \
        X_4 = \begin{pmatrix}
            5 & 6 \\
            10 & 6 \\
        \end{pmatrix}, \
        X_5 = \begin{pmatrix}
            9 & 5 \\
            6 & 2 \\
        \end{pmatrix}.
    \]
    Observe that $X_2 = 4 X_0$, $X_4 = 5 X_0$.
\end{example*}
Our goal is to analyze the limiting distribution of this random walk and to study the rate at which it converges to that distribution, particularly when the elements $A$ and $B$ are chosen uniformly at random. We prove the following somewhat surprising results.

The random walk $(X_k)_{k \in \N_0}$ exhibits different behavior depending on whether the number of steps is even or odd (as hinted at in the example above). In fact, as we will explain, it suffices to focus on what happens after an even number of steps. We first show that the distribution of $X_{2k}$ is supported on the line $\lin\{[A,B]\}$ and converges (in total variation distance\footnote{The total variation distance between probability measures $\mu$ and $\nu$ is $\dist(\mu, \nu) = \frac12 \norm{\mu - \nu}_1$.}) to a distribution $\nu_{2k}$ on this line in a \emph{pre-cutoff} manner at $k \approx p$.

\begin{theorem}
    \label[theorem]{thm:pre-cutoff}
    For every $\delta, \epsilon > 0$, there are constants $C, c > 0$ such that the following holds. Let $p > C$, let $A,B \in \slfrak_2(\F_p)$ be uniformly random, and let $(X_k)_{k \in \N_0}$ be the associated random Lie bracket. Let $\mu_{2k}$ be the distribution of $X_{2k}$. There is a distribution $\nu_{2k}$ that is uniform on a coset of a subgroup of $\F_p^\ast = \lin\{[A,B]\} \setminus \{ 0 \}$ such that with probability at least $1 - \delta$:
    \begin{enumerate}
        \item For $k < cp$, we have $\dist(\mu_{2k}, \nu_{2k}) > 1 - \epsilon$.
        \item For $k > Cp$, we have $\dist(\mu_{2k}, \nu_{2k}) < \epsilon$.
    \end{enumerate}
\end{theorem}

After $2k+1$ of steps, the distribution $\nu_{2k}$ equally splits into two distributions, one supported on the line $\lin\{[A,[A,B]]\}$ and the other on the line $\lin\{[B,[A,B]]\}$.

\begin{example*}
Consider the example above. Take $k=50$ and generate 1000 independent samples of the random Lie bracket $X_{100}$. The following table shows the sample distribution:
    \[
    \begin{pmatrix}
        [A,B] & 5 [A,B] & 4 [A,B] & 3 [A,B] & 9 [A,B] \\
        189 & 185 & 193 & 209 & 224
    \end{pmatrix}
    \]
In this case, the distribution $\nu_{100}$ (indeed, any $\nu_{2k}$ with $k>0$) is uniform over the five listed multiples of $[A,B]$. The total variation distance between the sample distribution and $\nu_{100}$ is $0.033$.
\end{example*}

For randomly chosen $A,B$, the distribution $\nu_{2k}$ is \emph{uniform on the whole of} $\lin\{[A,B]\} \setminus \{ 0 \}$ with a certain probability that depends on the prime divisors of $p-1$.

\begin{theorem}
Let $(s, t)$ be a subinterval of
\[
\mathcal I = 
% \left\{ \prod_{p \in P} \left( 1 - \frac{1}{p^2} \right) \mid 2 \in P \subseteq \P \right\} =
\left[ \frac{1}{\zeta(2)}, \frac23 \right] \cup \left[ \frac{9}{8 \zeta(2)}, \frac34 \right]
\approx [ 0.6079, 0.6667 ] \cup [0.6839, 0.7500].
\]
Then there is a set of primes $p$ of positive lower density with the following property. Let $A, B$ be uniformly random in $\slfrak_2(\F_p)$, and let $\nu$ be the uniform distribution on $\lin\{[A,B]\} \setminus \{ 0 \}$. Then $\P_{A,B}(\nu_{2k} = \nu) \in (s, t)$ for all $k \geq 0$.\footnote{The probability $\P_{A,B}(\nu_{2k} = \nu)$ is independent of $k$, as $\nu_{2k}$ is supported on cosets of the same subgroup of $\lin \{ [A,B] \} \setminus \{ 0 \}$. Let $\nu(p)$ denote this probability. The theorem and its proof imply that the set of limit points of the sequence $\nu(p)$ as $p$ varies is exactly $\mathcal I$.}
\end{theorem}

Based on our inspection of the random Lie bracket on $\slfrak_2(\F_p)$ driven by random generators $A$ and $B$, we can prove an additional result in this setting. Although the random walk requires roughly $p$ steps to approach its limiting distribution, the two elements $A$ and $B$ generate the entire Lie algebra $\slfrak_2(\F_p)$ much more quickly. In fact, we will show that, with high probability, every element of $\slfrak_2(\F_p)$ can be reached after only $O(\log p)$ steps.

We measure steps in a weighted way using the notion of the diameter of a Lie algebra as defined in \cite{dona2023sum}. Let $S$ be a subset of a Lie algebra $\gfrak$. The set $S$ is called \emph{symmetric} if $0 \in S$ and $S = -S$. For a symmetric set $S$, we inductively define the set of weighted $k$-balls as
\[
    S^1 = S, \qquad S^k = \bigcup_{0 < j < k} \left((S^j + S^{k-j}) \cup [S^j, S^{k-j}]\right) \quad (k \geq 2),
\]
where we have denoted $X+Y = \{ x + y \mid x \in X, y \in Y\}$ and $[X,Y] = \{[x,y] \mid x \in X, y \in Y\}$ for subsets $X,Y \subseteq \gfrak$.
If $S$ is a symmetric generating set of a Lie algebra $\gfrak$ over $\F_p$, then the \emph{diameter} $\diam(\gfrak, S)$ of $\gfrak$ with respect to $S$ is the smallest $k$ such that $S^k = \gfrak$.

For any classical Lie algebra $\gfrak$ over $\F_p$, Dona proves in \cite{dona2023sum} that there is a constant $C = O(\dim(\gfrak)^2\log\dim(\gfrak))$ such that
\[
    \diam(\gfrak, S) \leq \left(\log\abs{\gfrak}\right)^C
\]
for any symmetric generating set $S$ of $\gfrak$. For the case of $\slfrak_2(\F_p)$, this gives a bound of the form $(\log p)^{O(1)}$. We will now improve this bound to linear in $\log p$ for the case when the generating set is chosen uniformly at random.

\begin{theorem}
    \label[theorem]{thm:diameter}
    There is a (possibly empty) small set\footnote{A \emph{small set} is a set of positive integers whose sum of reciprocals converges.} of primes $T$ with the following property. For any $\epsilon > 0$ there is a $C$ such that for all primes $p > C$ not in $T$, we have
    \[
        \Pr_{A, B}\left(\diam\left(\slfrak_2(\F_p), S\right) \leq 120\log p + 8\right) \geq 1 - \epsilon,
    \]
    where $S = \{0, \pm A, \pm B\}$ for uniformly random $A, B \in \slfrak_2(\F_p)$.
\end{theorem}

The situation is very similar to what happens in groups: the diameter of $\SL_2(\F_p)$ is $(\log p)^{O(1)}$ for any generating set, and it is $O(\log p)$ for a random generating set \cite[Corollary 6.5]{helfgott2008growth}.

\subsection*{Reader's guide}

The ideas behind these results can be outlined as follows. First, we analyze the initial steps of the random walk and observe that odd and even steps exhibit distinct behaviors. This observation allows us to reduce the problem to a random walk on the cyclic group $\F_p^\ast$ of order $n = p - 1$ (see \Cref{sec:translating_to_cyclic_group}). The key parameters of this random walk are the coefficients of the Gram matrix associated with the Hilbert–Schmidt inner product on $\slfrak_2(\F_p)$. We then study the distribution of these parameters and demonstrate that they are nearly uniform (see \Cref{sec:gram_matrix_fibers}). As a result, the problem reduces to analyzing a random walk on the cyclic group $\Z_n$ with almost uniformly random generators, which can be approached using Fourier analysis on cyclic groups (see \Cref{sec:overview_of_fourier_analysis}). We identify the characters that contribute to the limiting distribution $\nu_{2k}$ and relate the probability of converging to the uniform distribution to a $\zeta(2)$–type condition concerning the prime divisors of $n = p-1$ (see \Cref{sec:probability_of_convergence_to_uniform}). After that, we establish the pre-cutoff\footnote{The term \emph{cutoff} describes the stronger property that $t_{\text{mix}}(\epsilon)/t_{\text{mix}}(1 - \epsilon) \to 1$ for any $0 < \epsilon < 1/2$, where $t_{\text{mix}}(x) = \min \{ k \in \N \mid \dist(\mu_{2k}, \nu_{2k}) < x \}$. This phenomenon cannot occur on a cyclic group with a bounded number of generators by \cite{nocutoff}.} phenomenon for the random Lie bracket (see \Cref{sec:precutoff}). Upper bounds on the total variation distance in the generic case were first obtained in \cite{hildebrand1994random}, and lower bounds were established in \cite{hildebrand2005survey} (see Theorem 2 and the exercises following it). Here, we provide simpler proofs tailored to our specific requirements. For the upper bound, we adapt recent elementary methods based on Fourier analysis from \cite{nocutoff}. These techniques alone yield only a $1/2-\epsilon$ lower bound, as we discuss. We therefore streamline the probabilistic argument from \cite{hildebrand2005survey} in our context to obtain the stronger $1 - \epsilon$ bound. Lastly, we discuss the linear bound of $O(\log p)$ on the diameter with respect to generic generators (see \Cref{diameter section}). Our result here relies on embedding a variant of the well-known Chung-Diaconis-Graham affine random walk on $\F_p$ \cite{chung1987random} into the random Lie bracket, which allows us to leverage the results of \cite{breuillard2022cut} on covering times of such walks.

\subsection*{Acknowledgments}

We thank Daniele Dona for suggesting the use of the results in \cite{breuillard2022cut} to establish diameter bounds. We also thank Martin Hildebrand for drawing our attention to the total variation distance bounds for random walks on cycles in \cite{hildebrand2005survey}. Finally, we thank the anonymous referee for their many helpful comments and suggestions.

% % to remove in final version!
% \setcounter{tocdepth}{1}
% \tableofcontents

\section{Translating the walk to a cyclic group}
\label[section]{sec:translating_to_cyclic_group}

\subsection{Inner product and adjoint operators}

The vector space $\slfrak_2(\F_p)$ comes equipped with the inner product $\langle A, B \rangle = \tr(A B)$. Note that $\langle A, A \rangle = \tr(A^2)$, and since $A^2 + \det(A) = 0$ for a matrix in $\slfrak_2(\F_p)$, we have $\langle A, A \rangle = - 2 \det(A)$.\footnote{Throughout the paper we assume $p > 2$.}

% An important property of this inner product is that adjoint operators are skew-adjoint with respect to it: $\ad_A^\ast = - \ad_A$. This is because for any $A, B, C \in \slfrak_2(\F_p)$, we have $\langle \ad_A B, C \rangle = \tr(BCA) - \tr(BAC) = - \langle B, \ad_A C \rangle$.

\begin{lemma}[compare with Proposition 6.2 in \cite{cantor2025two}]
\label[lemma]{lemma:ad^2}
For any two elements $A, B \in \slfrak_2(\F_p)$, we have
\[
[A, B, B] = 2 \langle B, B \rangle A - 2 \langle A, B \rangle B.
\]
\end{lemma}
\begin{proof}
The characteristic polynomial of $A$ is $A^2 + \det A = 0$, hence $A^3 - \frac12 \langle A, A \rangle A = 0$. Multilinearize this identity by considering it in the Lie algebra $\slfrak_2(\F_p) \otimes_{\F_p} \F_p[T]/(T^2)$ and using it with the element $TA + B$. Inspecting the $T$-term, we obtain
\[
AB^2 + BAB + B^2 A - \frac12 ( 2 \langle A, B \rangle B + \langle B, B \rangle A ) = 0.
\]
Taking into account that $B^2 = \frac12 \langle B, B \rangle$, we get
\[
\langle B, B \rangle A - 2 \langle A, B \rangle B + 2 BAB = 0.
\]
The last term of the sum is $2BAB = 2B[A, B] + 2 B^2 A = 2 B[A, B] + \langle B, B \rangle A$, hence
\begin{equation}
    \label{eq:ad^2_proof}
    2 \langle B, B \rangle A - 2 \langle A, B \rangle B = - 2 B [A, B].
\end{equation}
Finally, we have $B [A, B] = BAB - B^2 A = BAB - A B^2 = [B, A] B$, therefore
\[
[A, B, B]
= [A, B] B - B [A, B]
= - 2 B [A, B],
\]
which is the right-hand side of \eqref{eq:ad^2_proof}.
\end{proof}

We clearly have $\ad_A A = 0$, $\ad_A B = [A, B]$, and it follows from \Cref{lemma:ad^2} that $\ad_A [A, B] = - 2 \langle A, B \rangle A + 2 \langle A, A \rangle B$. Symmetrically, we have $\ad_B A = - [A, B]$, $\ad_B B = 0$, and $\ad_B [A, B] = - 2 \langle B, B \rangle A + 2 \langle A, B \rangle B$. Thus, if $A, B$ is a generating pair of the Lie algebra $\slfrak_2(\F_p)$, the matrices $A, B, [A, B]$ generate it as a vector space. In this basis, the adjoint operators $\ad_A$ and $\ad_B$ can be represented by the matrices
\[
    \ad_A = 
    \begin{pmatrix}
        0 & 0 & -2 \langle A, B \rangle \\
        0 & 0 & 2 \langle A, A \rangle \\
        0 & 1 & 0
    \end{pmatrix},
    \qquad
    \ad_B = 
    \begin{pmatrix}
        0 & 0 & -2 \langle B, B \rangle \\
        0 & 0 & 2 \langle A, B \rangle \\
        -1 & 0 & 0
    \end{pmatrix}.
\]
To simplify notation we shall write $\alpha = 2 \langle A, B \rangle$, $\beta = 2 \langle A, A \rangle$ and $\gamma = 2 \langle B, B \rangle$.

\subsection{Odd/even number of steps}

The random walk $(X_k)_{k \in \mathbb{N}_0}$ begins at $X_0 = [A, B]$. After the first step, $X_1$ is uniformly distributed among the elements $\ad_A X_0 = -\alpha A + \beta B$ and $\ad_B X_0 = -\gamma A + \alpha B$. At the second step, $X_2$ is uniformly distributed among $\beta [A, B]$, $\alpha [A, B]$, $\alpha [A, B]$, and $\gamma [A, B]$. 

By induction, we observe that for all even indices, $X_{2k} \in \lin\{[A, B]\}$, where $\lin\{[A, B]\}$ denotes the line spanned by $[A, B]$. Similarly, for all odd indices, $X_{2k+1} \in \lin\{-\alpha A + \beta B\} \cup \lin\{-\gamma A + \alpha B\}$. Given the distribution of $X_{2k}$ on the line $\lin\{[A, B]\}$, it follows that $X_{2k+1}$ is distributed proportionally (scaled by a factor of $1/2$) across the lines $\lin\{-\alpha A + \beta B\}$ and $\lin\{-\gamma A + \alpha B\}$. 

Thus, to understand the behavior of our random walk, it suffices to focus on the distribution after an even number of steps. This allows us to reduce the problem to analyzing how the distribution evolves along the line $\lin\{[A, B]\}$.

The sequence $(X_{2k})_{k \in \N}$ is equivalent to a random walk on $\F_p$, where we start at the element $1 \in \F_p$ and, at each step, multiply by $\alpha$ with probability $1/2$, or by $\beta$ or $\gamma$ each with probability $1/4$. Provided that $\alpha$, $\beta$, and $\gamma$ are nonzero, this random walk occurs on the group $\F_p^\ast$. By fixing an isomorphism $\log \colon \F_p^\ast \to \Z_n$ where $n = p - 1$, we can translate the problem to a random walk on the cyclic group $\Z_n$.

\section{Distribution of the Gram matrix}
\label[section]{sec:gram_matrix_fibers}

In order to analyze our random walk, we need to determine how the parameters $\alpha = 2\langle A, B \rangle$, $\beta = 2 \langle A, A \rangle$, and $\gamma = 2 \langle B, B \rangle$ are distributed in $\F_p$ as $A,B$ vary across $\slfrak_2(\F_p)$. These parameters are essentially coefficients of the Gram matrix of the inner product on $\slfrak_2(\F_p)$:
\[
    G_{A,B} = 
    \begin{pmatrix}
        \langle A, A \rangle & \langle A, B \rangle \\
        \langle A, B \rangle & \langle B, B \rangle
    \end{pmatrix}
    = \frac12 
    \begin{pmatrix}
        \beta & \alpha \\
        \alpha & \gamma
    \end{pmatrix}.
\]
We thus consider the Gram map to the set of all symmetric matrices,
\[
    G \colon \slfrak_2(\F_p)^2 \to \Sym_2(\F_p), 
    \quad (A, B) \mapsto G_{A,B}.
\]
The orthogonal group $O(3)$ over $\F_p$ acts naturally on the space $\slfrak_2(\F_p)$ (see \cite{clark2013quadratic}), with Witt decompositon $\langle e, f \rangle \oplus \langle h \rangle$, where $e = E_{12}$, $f = E_{21}$, and $h = E_{11} - E_{22}$. The discriminant of the form is $-2$. This action induces an action on pairs $(A, B) \in \slfrak_2(\F_p)^2$. The Gram map $G$ is invariant under this action. 

\subsection{Fibers of the Gram map}

\begin{lemma}
The number of elements $A \in \slfrak_2(\F_p)$ with $\langle A, A \rangle = a$ is equal to $p(p + \eta(a/2))$, where $\eta$ is the quadratic character on $\F_p$.
\end{lemma}
\begin{proof}
Write $A = x e + y f + z h$ with $x,y,z \in \F_p$. The condition $\langle A, A \rangle = a$ is equivalent to $2xy + 2z^2 = a$. This is an affine quadric. If $y \neq 0$, we can uniquely solve for $x$, giving $(p-1)p$ solutions. If $y = 0$, we have $z^2 = a/2$, which has $1 + \eta(a/2)$ solutions for $z$, and $x$ is arbitrary, giving $p(1+\eta(a/2))$ solutions.
\end{proof}

\begin{proposition}
\label[proposition]{prop:gram_matrix_fibers}
Let $X \in \Sym_2(\F_p)$. Then
\[
|G^{-1}(X)|
=
\begin{cases}
    p^3 - p & \rank X = 2 \\
    2p^3 + p^2 - p & \rank X = 1, \ \eta(X_{11}/2), \eta(X_{22}/2) \geq 0 \\
    p^2 - p & \rank X = 1, \ \eta(X_{11}/2) = -1 \text{ or } \eta(X_{22}/2) = -1 \\
    p^3 + p^2 - p & \rank X = 0.
\end{cases}
\]
\end{proposition}
\begin{proof}
Each fiber of $G$ is a disjoint union of orbits under the action of $O(3)$. The size of the orthogonal group is $|O(3)| = 2 p (p^2 - 1)$.

Suppose first that $\rank X = 2$.
Each pair $(A,B) \in \slfrak_2(\F_p)^2$ with $G_{A,B} = X$ determines a nondegenerate quadratic subspace $\langle A, B \rangle$ of $\slfrak_2(\F_p)$. For any other pair $(A', B')$ with $G_{A',B'} = X$, we have an isometry mapping $A \mapsto A', B \to B'$. By Witt's extension lemma, this isometry can be extended to an action of $O(3)$ on $\slfrak_2(\F_p)$ that maps $(A, B)$ to $(A', B')$. Thus the fiber of $G$ over $X$ is a single orbit of $O(3)$. The stabilizer $\Stab_{O(3)}(A,B)$ fixes $\langle A,B \rangle$ pointwise, hence it also preserves the line $\langle A, B \rangle^\perp$, on which it can only act as $\pm \id$. Hence $|\Stab_{O(3)}(A,B)| = 2$ and so the orbit of $O(3)$ on $(A,B)$ is of size $p^3 - p$.

Assume now that $\rank X = 0$.
Suppose $(A,B)$ is a pair with $G_{A,B} = 0$. Since the Witt index of the quadratic space $\slfrak_2(\F_p)$ is $1$, we must have $\dim \langle A, B \rangle < 2$. Assume $(A,B) \neq (0,0)$, which forms a single orbit of size $1$. Thus $(A,B) = (\lambda Z, \mu Z)$ for some $(\lambda, \mu) \in \F_p^2 \setminus \{ 0 \}$ and $0 \neq Z \in \slfrak_2(\F_p)$ with $\langle Z, Z \rangle = 0$. By Witt extension lemma, this pair maps to a pair with $Z = e$ under $O(3)$. Two such pairs can further be mapped to each other by an $O(3)$ map if and only if their $(\lambda, \mu)$ vectors are dependent. Thus a set of orbit representatives is parameterized by vectors in $\P^1$, so there are $p+1$ orbits in total. Each such orbit has size $|O(3).e|$, which is the number of nonzero isotropic vectors in $\slfrak_2(\F_p)$. There are a total of $p^2 - 1$ of these by the previous lemma. In total, we obtain $1 + (p+1)(p^2 - 1) = p^3 + p^2 - p$ pairs in the fiber over $0$.

Consider finally the option when $\rank X = 1$. Suppose $(A,B)$ is a pair with $G_{A,B} = X$.
\begin{description}
    \item[$\dim \langle A, B \rangle = 1$]
    This means $(A,B) = (xZ, yZ)$ for some $Z \in \slfrak_2(\F_p)$ and $x,y$ not both $0$. The Gram matrix of this pair is
    \[
        X = \langle Z,Z \rangle \begin{pmatrix}
            x^2 & xy \\ xy & y^2
        \end{pmatrix},
    \]
    so $\langle Z, Z \rangle \neq 0$. Let $(C,D)$ be another pair with $G_{C,D} = X$. Thus $(C,D) = (\alpha W, \beta W)$ for some $\alpha, \beta$ not both zero. Comparing the Gram matrices, it follows that $\langle W, W \rangle / \langle Z, Z \rangle$ is a nonzero square $u^2 \in \F_p$. Hence $W$ can be mapped by an element of $O(3)$ to $uZ$, and the pair $(C,D)$ gets mapped to $(\alpha u Z, \beta u Z)$. Again comparing the Gram matrices gives $\alpha u = \pm x$ and $\beta u = \pm y$, where nonequal signs can only occur if $\alpha \beta = 0$. Hence $(C,D) = (xZ, yZ) = (A,B)$ or $(C,D)=(-A,-B)$ when $\alpha \beta \neq 0$, and $(C,D) = (0, \pm yZ) = (0, \pm B)$ or $(\pm xZ, 0) = (\pm A, 0)$ when $\alpha \beta = 0$. In all cases, an extra application of an element of $O(3)$ identifies elements with distinct signs. Therefore $(C,D)$ is in the orbit of $(A,B)$, and hence $(A,B)$ has a single orbit under $O(3)$. The stabilizer of any such point are the transformations in $O(3)$ that fix $Z$. Since $Z$ is anisotropic, we have two options for the complement: either it is split or non-split, depending on whether the discriminant $\Delta(\langle Z \rangle^\perp) = -2/\langle Z, Z \rangle$ is $-1$ modulo squares. In the split case, the stabilizer has size $2(p-1)$, and in the non-split case, it has size $2(p+1)$. Thus the stabilizer is of size $2(p - \eta(\langle Z, Z \rangle / 2))$. The number of pairs in the fiber over $X$ is thus $p(p + \eta(\langle Z, Z \rangle / 2))$. These orbits appear whenever $X$ is of rank $1$.

    \item[$\dim \langle A, B \rangle = 2$] 
    In this case, we can find a nonzero vector $C \in \langle A, B \rangle$ that is orthogonal to $A$ and $B$. Hence $C$ is isotropic and we can map it to $e$ under $O(3)$. Then $A, B \in \langle e \rangle^\perp = \langle e, h \rangle$. Thus $A = x_0 e + x h$ and $B = y_0 e + y h$ with not both $x,y$ zero. Assuming $y \neq 0$, an orthogonal transformation maps this pair into $(e + xh, yh)$. On the other hand, if $y = 0$, we can map it to $(xh, e)$. In both cases, the values of $x,y$ are uniquely determined by $X$ (the option $-x, -y$ is equivalent to $(A,B)$ by the flip along $\langle e, f \rangle$). Hence we again have a unique orbit over any $X$. Since the stabilizer of $(e, f)$ is trivial, the orbit is of size $2p(p^2 - 1)$. These orbits only appear when $\eta(X_{11}/2), \eta(X_{22}/2) \geq 0$ (not both zero).
\end{description}
Summing up the contributions from the two cases completes the proof.
\end{proof}

\begin{corollary}
    \label[corollary]{cor:probability_of_nonzero_entries}
Let $A,B$ be uniformly random in $\slfrak_2(\F_p)$. Then $G_{A,B}$ has all entries nonzero with probability at least $1 - 9/p$.
\end{corollary}
\begin{proof}
Each fiber of the Gram map is of size at most $3p^3$. The probability that $G_{A,B}$ has a zero entry can thus be upper bounded by the union bound as
\[
    \frac{1}{p^6} \sum_{X \in \Sym_2(\F_p) \setminus \Sym_2(\F_p^\ast)} \abs{G^{-1}(X)} \leq \frac{3p^3 \cdot 3p^2}{p^6} = \frac{9}{p}. \qedhere
\]
\end{proof}

\subsection{Distribution of parameters governing the random walk}

Let $\mathcal U = G^{-1}(\Sym_2(\F_p^\ast))$ be the set of pairs $(A,B)\in\slfrak_2(\F_p)^2$ whose Gram matrix $G_{A,B}$ has all entries nonzero. For any $(A,B)\in\mathcal U$, the random walk on $\slfrak_2(\F_p)$ after an even number of steps reduces to a random walk on the cyclic group $\Z_{p-1}$ as in \Cref{sec:translating_to_cyclic_group}. In this case, the behavior of the walk is determined by two parameters $a,b$ extracted from $G_{A,B}$ via the map
\[
    P \colon \Sym_2(\F_p^\ast) \to \Z_{p - 1}^2,
    \quad 
    X = \frac12 \begin{pmatrix}
        \beta & \alpha \\
        \alpha & \gamma
    \end{pmatrix}
    \mapsto \left( \log(\alpha/\beta), \log(\alpha/\gamma) \right).
\]
Let $F = P \circ G \colon \mathcal U \to \Z_{p - 1}^2$ be the composite map, so that $(a,b) = F(A,B)$. Let $\pi$ be the pushforward measure of the uniform distribution on $\slfrak_2(\F_p)^2$ under $F$. In other words, $\pi(S) = \abs{F^{-1}(S)}/p^6 = \P_{A,B}(P(G_{A,B}) \in S)$ for any subset $S \subseteq \Z_{p - 1}^2$. Denote the uniform distribution on $\Z_{p - 1}^2$ by $\nu_{\Z_{p - 1}^2}$. As $p$ grows to infinity, the distributions $\pi$ and $\nu_{\Z_{p - 1}^2}$ become close to each other.

\begin{proposition}
    \label[proposition]{distribution of a and b}
    For any subset $S \subseteq \Z_{p - 1}^2$, we have
    \[
        \abs{\pi(S) - \nu_{\Z_{p - 1}^2}(S)} \leq \frac{6}{p}.
    \]
\end{proposition}

To prove this proposition, we first examine the sizes of fibers of the map $F$.

\begin{lemma}
    For any $x = (a, b) \in \Z_{p - 1}^2$ we have 
    \[
        \abs{\abs{F^{-1}(x)} - \frac{p^6}{(p - 1)^2}} \leq 6p^3.
    \]
\end{lemma}
\begin{proof}
    We have $\abs{P^{-1}(x)} = p - 1$. First, suppose that $a + b = 0$. In this case we have $\alpha^2 = \beta\gamma$, so all matrices $X \in P^{-1}(x)$ have rank $1$. Half of these satisfy $\eta(X_{11}/2) = \eta(X_{22}/2) = 1$, and the other half satisfy $\eta(X_{11}/2) = \eta(X_{22}/2) = -1$. It follows from \Cref{prop:gram_matrix_fibers} that
    \[
        \abs{F^{-1}(x)} = \frac{p - 1}{2}(2p^3 + p^2 - p) + \frac{p - 1}{2}(p^2 - p) = p^4 - 2p^3 + p.
    \]
    We compute $0 \leq p^6 - (p^4 - 2p^3 + p)(p - 1)^2 \leq 6p^3(p - 1)^2$. Dividing by $(p - 1)^2$ completes the proof for this case.

    Now, suppose that $a + b \neq 0$. In this scenario, all matrices $X \in P^{-1}(x)$ have rank $2$. Again, by \Cref{prop:gram_matrix_fibers}, the size of the fiber $F^{-1}(x)$ is $(p-1)(p^3 - p) = p^4 - p^3 - p^2 + p$. As in the previous case, this quantity is within $6p^3$ of $p^6/(p - 1)^2$.
\end{proof}

\begin{proof}[Proof of \Cref{distribution of a and b}]
    We have 
    \[
        \abs{\pi(S) - \nu_{\Z_{p - 1}^2}(S)} = \abs{\frac{1}{p^6} \sum_{x \in S} \abs{F^{-1}(x)} - \frac{\abs{S}}{(p - 1)^2}
        } \leq \frac{1}{p^6} \sum_{x \in S} \abs{\abs{F^{-1}(x)} - \frac{p^6}{(p - 1)^2}}.
    \]
    The last sum can be upper bounded by the last lemma as $6p^3\abs{S} / p^6 \leq 6/p$.
\end{proof}

We will need the following corollary multiple times in the rest of the paper.

\begin{corollary}
    \label[corollary]{cor:pushforward_vs_uniform}
    Let $S \subseteq \slfrak_2(\F_p)^2$ and $T \subseteq \Z_{p - 1}^2$ with $S \cap \mathcal U = F^{-1}(T)$. Then 
    \[
        \abs{\frac{\abs{S}}{p^6} - \frac{\abs{T}}{(p - 1)^2}} < 15/p.
    \]
\end{corollary}
\begin{proof}
    Since $F^{-1}(T) \subseteq S \subseteq F^{-1}(T) \cup (\slfrak_2(\F_p)^2 \setminus \mathcal U)$, we have 
    \[
        \abs{\frac{\abs{S}}{p^6} - \frac{\abs{F^{-1}(T)}}{p^6}} \leq \frac{\abs{\slfrak_2(\F_p)^2 \setminus \mathcal U}}{p^6} \leq \frac{9}{p}
    \]
    by \Cref{cor:probability_of_nonzero_entries}. On the other hand, we have 
    \[
        \abs{\frac{\abs{F^{-1}(T)}}{p^6} - \frac{\abs{T}}{(p - 1)^2}} = \abs{\pi(T) - \nu_{\Z_{p - 1}^2}(T)} \leq \frac{6}{p}
    \]
    by \Cref{distribution of a and b}. The result now follows by the triangle inequality.
\end{proof}

We shall apply the above corollary as follows. Suppose that for every $(A,B)\in \mathcal U$, the random walk on $\slfrak_2(\F_p)$ exhibits a property $\mathcal P$ if and only if the corresponding parameters $(a,b) = F(A,B) \in \Z_{p-1}^2$ satisfy a condition $\mathcal C$. Then, for uniformly random $A,B \in \slfrak_2(\F_p)$, the probability that the walk exhibits property $\mathcal P$ differs by at most $15/p$ from the probability that uniformly random $(a,b) \in \Z_{p-1}^2$ satisfy $\mathcal C$.

\section{A brief overview of Fourier analysis on cyclic groups}
\label[section]{sec:overview_of_fourier_analysis}

We give a brief overview of Fourier analysis on cyclic groups, following the exposition in \cite{diaconis1988group}.

\subsection{Setup}

Let $S$ be a subset of a cyclic group $\Z_n$ equipped with a probability measure $p \colon S \to \C$. Let $(Y_k)_{k \in \N_0}$ be the random walk on $\Z_n$ starting at $0$ with transition probabilities given by $p(x,y) = p(y - x)$. Let $\sigma_k$ be the probability distribution of the walk after $k$ steps, i.e., $\sigma_k(x) = \Pr(Y_k = x)$ for $x \in \Z_n$. This distribution is a function in the vector space $\lspace(\Z_n)$ of all complex valued functions on $\Z_n$. 

\subsection{Markov operator}

Let $M \colon \lspace(\Z_n) \to \lspace(\Z_n)$ be the operator defined by
\[
    (Mf)(x) = \sum_{s \in S} p(s)f(x - s).
\]
The distribution $\sigma_k$ can be written as
\[
    \sigma_k(x) = \sum_{s \in S} p(s)\sigma_{k-1}(x - s) = (M\sigma_{k-1})(x),
\]
and so $\sigma_k = M^k \sigma_0 = M^k 1_0$. In order to understand high powers of the Markov operator, we need to analyze its eigenvalues and eigenfunctions. For any character $\chi$ of the group $\Z_n$, we have
\[
    (M\chi)(x) = \left(\sum_{s \in S} p(s)\chi(-s)\right)\chi(x),
\]
so $\chi$ is an eigenfunction of $M$ with eigenvalue $\lambda = \sum_{s \in S} p(s)\chi(-s)$. Characters of $\Z_n$ are of the form $\chi_j(x) = \omega^{jx}$ for $j \in \Z_n$, where $\omega = e^{2\pi i/n}$, and they form a basis of $\lspace(\Z_n)$. 

\subsection{Limiting distribution}

The initial distribution $1_0$ can be expressed in terms of the basis of characters as $1_0 = \frac{1}{n} \sum_{j = 0}^{n - 1}\chi_j$. It follows that the distribution after $k$ steps is given by
\[
    \sigma_k = \frac{1}{n}\sum_{j = 0}^{n - 1}\lambda_j^k\chi_j,
\]
where $\lambda_j = \sum_{s \in S} p(s)\chi_j(-s)$ is the eigenvalue corresponding to $\chi_j$. We can decompose $\sigma_k$ into two components $\sigma_k = \phi_k + \omega_k$, where 
\[
    \phi_k = \frac{1}{n}\sum_{|\lambda_j| = 1} \lambda_j^k\chi_j,
    \qquad
    \omega_k = \frac{1}{n}\sum_{|\lambda_j| < 1} \lambda_j^k\chi_j.
\]
As $k$ grows to infinity, the contribution of $\omega_k$ diminishes to zero, making $\sigma_k$ increasingly close to $\phi_k$.\footnote{Note that $\phi_k$ may not converge as $k \to \infty$, nor is it guaranteed to be uniform.} This proximity can be quantified in terms of the spectral radius $\rho = \max\{|\lambda_j| \mid |\lambda_j| < 1\}$ of the Markov operator. 

\begin{lemma}[Lemma 3 in \cite{nocutoff}]
    \label[lemma]{distance to limiting distribution}
    The distance to the limiting distribution satisfies
    \[
    \rho^{2k} \leq \norm{\sigma_k - \phi_k}_1^2 \leq \sum_{|\lambda_j| < 1}\abs{\lambda_j}^{2k}.
    \]
\end{lemma}
\begin{proof}
The characters $\chi_j$ form an orthogonal basis of $\lspace(\Z_n)$ and have norm $\sqrt{n}$. Thus
\[
\norm{\omega_k}_1^2 \leq n\norm{\omega_k}_2^2 = \frac{1}{n}\sum_{|\lambda_j| < 1}\norm{\lambda_j^k\chi_j}_2^2 = \sum_{|\lambda_j| < 1}\abs{\lambda_j}^{2k},
\]
proving the second inequality. For the first inequality, let $\chi$ be a character of $\Z_n$ corresponding to an eigenvalue with absolute value $\rho$. Since $\norm{\chi}_\infty = 1$, we have
\[
\norm{\omega_k}_1 
\geq \abs{\langle\omega_k, \chi\rangle} 
= \frac{1}{n}\rho^k\abs{\langle\chi, \chi\rangle} 
= \rho^k. \qedhere
\]
\end{proof}

\section{The limiting distribution}
\label[section]{sec:probability_of_convergence_to_uniform}

We are now ready to analyze our random Lie bracket in more detail. We have already seen that the distribution of the random walk after an even number of steps is equivalent to a random walk on the cyclic group $\F_p^\ast$ starting at $1$ and evolving by multiplication with $\alpha, \beta, \gamma$ (all nonzero) with probabilities $1/2,1/4,1/4$. 

\subsection{Log}

Fix an isomorphism $\log: \F_p^\ast \to \Z_n$ and let $\alpha' = \log \alpha$, $\beta' = \log \beta$, and $\gamma' = \log \gamma$. Under this mapping, the random walk on $\F_p^\ast$ translates to a random walk on $\Z_n$, starting at $0 = \log 1 \in \Z_n$. At each step, the walk adds $\alpha'$ with probability $1/2$, or $\beta'$ or $\gamma'$ each with probability $1/4$. Let $S' = \{ \alpha', \beta', \gamma' \}$ be the set of steps of the walk with probability measure $p(\alpha') = 1/2$, $p(\beta') = p(\gamma') = 1/4$. The limiting distribution on $\Z_n$ is given by
\[
    \phi_k = \frac1n \sum_{|\lambda_j| = 1} \lambda_j^k \chi_j.
\]

\subsection{Contributing characters}

Let us examine which characters $\chi_j$ appear in the sum $\phi_k$. Certainly the trivial character $\chi_0$ will always contribute to the limiting distribution. The character $\chi_j$ will contribute if and only if

\[
    \abs{\lambda_j} = \abs{\sum_{s \in S'} p(s)\chi_j(-s)} = 1.
\]
As $\abs{\chi_j(-s)} = 1$ for each $s \in S'$, this holds if and only if $\chi_j(-\alpha') = \chi_j(-\beta') = \chi_j(-\gamma')$, which is equivalent to $\chi_j(\alpha' - \beta') = \chi_j(\alpha' - \gamma') = 1$. Introduce parameters $a = \alpha' - \beta'$ and $b = \alpha' - \gamma'$. 
Hence a character $\chi_j$ appears in the limiting distribution if and only if $a,b \in \ker \chi_j$, and the corresponding eigenvalue is $\lambda_j = \overline{\chi_j(\alpha')}$. The random walk is thus essentially determined by the parameters $a,b$.

\begin{lemma}
    \label[lemma]{equivalent characters}
    $\ker\chi_j = (n/\gcd(n, j))\Z_n$.
\end{lemma}
\begin{proof}
    An element $x \in \Z_n$ satisfies $x \in \ker\chi_j$ if and only if $n$ divides $jx$. We can write $j = dj'$ and $n = dn'$, where $d = \gcd(j, n)$. By canceling $d$, we see that this is equivalent to $n'$ dividing $j'x$ and hence to $n'$ dividing $x$. Therefore $\ker\chi_j = \frac{n}{d}\Z_n$ and $\abs{\ker\chi_j} = d$.
\end{proof}

\begin{lemma}
    \label[lemma]{lemma:contributing_characters}
    Let $a = \alpha' - \beta'$ and $b = \alpha' - \gamma'$. Then the characters appearing in $\phi_k$ are precisely those $\chi_j$ for which $j$ is divisible by $n/\gcd(a,b,n)$. The number of contributing characters is thus $\gcd(a,b,n)$.
    \end{lemma}
    \begin{proof}
    The character $\chi_j$ contributes to the limiting distribution if and only if $a,b \in \ker\chi_j = (n/\gcd(n,j)) \Z_n$. This is equivalent to $n/\gcd(n,j) \mid a,b$, which is the same as $n/\gcd(a,b,n) \mid \gcd(n,j)$, and this is further equivalent to $j$ being divisible by $n/\gcd(a,b,n)$.
\end{proof}

\begin{corollary}
    \label[corollary]{cor:uniform distribution}
    The limiting distribution $\phi_k$ is uniform if and only if $\gcd(a,b,n) = 1$.
\end{corollary}

In fact, we always get a uniform distribution on $\gcd(a,b,n)\Z_n$ provided the number of steps $k$ is divisible by $\gcd(a,b,n)$, and more generally a uniform distribution on a coset of $\gcd(a,b,n)\Z_n$.

\begin{lemma}
    \label[lemma]{lemma:distribution_nu}
    The distribution $\phi_k$ is uniform on the coset $k\alpha' + \gcd(a,b,n)\Z_n$.
    \end{lemma}
    \begin{proof}
    Let $d = \gcd(a,b,n)$. We have
    \[
        \phi_k = \frac{1}{n} \sum_{\frac{n}{d} \mid j} \overline{\chi_j(\alpha')}^k \chi_j.
    \]
    Any $x \in k\alpha' + d\Z_n$ can be written as $x = k\alpha' + md$ for some $m \in \Z$. Thus
    \[
        \overline{\chi_j(\alpha')}^k \chi_j(x) = \chi_j(-k\alpha')\chi_j(k\alpha' + md) = \chi_j(md) = 1
    \]
    for any $j$ appearing in the sum. Therefore $\phi_k(x) = d/n$, and so $\phi_k$ is uniform on $k\alpha' + d\Z_n$.
\end{proof}

\subsection{Randomly chosen parameters}

Suppose that the parameters $\alpha', \beta', \gamma'$ are chosen uniformly at random from $\Z_n$. In this case, the parameters $a = \alpha' - \beta'$ and $b = \alpha' - \gamma'$ are also uniformly distributed in $\Z_n$. Thus the probability that a character $\chi_j$ contributes to the limiting distribution is equal to $(|\ker\chi_j| / n)^2$. Taking $j = 2$, we see that the probability that $\chi_{n/2}$ contributes is $1/4$. It follows that the limiting distribution $\phi_k$ is uniform with probability at most $3/4$. We will show that for any given $\epsilon > 0$, there is a positive proportion of primes for which the probability that any other character contributes to the limiting distribution is less than $\epsilon$. In fact, we can compute the exact probability that the limiting distribution is uniform in terms of the prime divisors of $n$.

\begin{lemma}
    \label[lemma]{probability of convergence}
    Let $a,b \in \Z_n$ be uniformly random. Then
    \[
        \Pr_{a,b}(\gcd(a,b,n) = 1) = \prod_{q \mid n}\left(1 - \frac{1}{q^2}\right),
    \] 
    where the product is over all primes $q$ dividing $n$.
\end{lemma}
\begin{proof}
    For each prime $q \mid n$, the probability that $q$ divides both $a$ and $b$ is $1/q^2$. Since the events for distinct primes are independent, the probability that no prime divisor of $n$ divides both $a$ and $b$ (which is exactly the probability that $\gcd(a,b,n) = 1$) is
    \[
    \prod_{q \mid n}\left(1 - \frac{1}{q^2}\right). \qedhere
    \]
\end{proof}

When the only small prime dividing $n$ is $2$, the probability that the limiting distribution is uniform is close to $3/4$. Conversely, when $n$ is divisible by all the primes up to some large number (‘‘divisible by all primes’’), the probability that the limiting distribution is uniform is close to
\begin{equation}
    \label{eq:1/zeta2}
    \prod_{p \in \P} \left(1 - \frac{1}{p^2}\right) = \frac{1}{\zeta(2)} \approx 0.6079.
\end{equation}
We will now give more precise estimates for the probability that the limiting distribution is uniform and show it can come arbitrarily close to almost all numbers between $1/\zeta(2)$ and $3/4$ apart from some subinterval.

\subsection{Probability of convergence to the uniform distribution}

\begin{lemma}
    \label[lemma]{determined factors}
    Let $X$ and $Y$ be two finite disjoint sets of odd primes. There exists a positive lower density\footnote{The lower density of a subset $Z$ of $\P$ is $\liminf_{n \to \infty} |Z \cap \P_{\leq n}|/|\P_{\leq n}|$.} subset of primes $p \in \P$ that satisfy
    \begin{enumerate}
        \item All primes from $X$ divide $p - 1$.
        \item No prime from $Y$ divides $p - 1$.
    \end{enumerate}
\end{lemma}
\begin{proof}
    Let $x = \prod_{q \in X}q$ and $y = \prod_{q \in Y}q$. Since $x$ and $y$ are coprime, the conditions $p \equiv 1 \pmod{x}$ and $p \equiv 2 \pmod{y}$ are equivalent to $p \equiv c \pmod{x y}$ for some $c \in \N$ by the Chinese remainder theorem. As $c$ is coprime to $x y$, the proportion of primes congruent to $c$ modulo $x y$ is $1 / \phi(x y) > 0$ (see \cite[Section VI.4]{serre2012course}). For any prime $p$ with this property, $p - 1$ is divisible by $x$ and is coprime to $y$.
\end{proof}

\begin{lemma}
    \label[lemma]{subset sum density}
    Let $X = \sum_{j \geq 1} x_j$ be the sum of a convergent series of positive numbers, where $x_k < \sum_{j \geq k + 1} x_j$ for each $k \in \N$. Then for any open interval $(s, t)$ contained in $(0, X)$, there is a finite subset $S \subseteq \N$ such that $\sum_{j \in S} x_j \in (s, t)$.
\end{lemma}
\begin{proof}
    Inductively define subsets $S_n \subseteq \N$ by
    \[
        S_1 = \begin{cases}
            \{ x_1 \} & x_1 < t \\
            \emptyset & \text{otherwise},
        \end{cases}
        \qquad
        S_{n+1} = \begin{cases}
            S_n \cup \{ n + 1 \} &  x_{n + 1} + \sum_{j \in S_n} x_j < t \\
            S_n & \text{otherwise}.
        \end{cases}
    \]
    Take $S = \bigcup_{n \geq 1} S_n$. We may assume that $t < X$ after possibly shortening the interval. Thus we can find some $k \in \N$ with $k \notin S$. Then
    \[
    \sum_{j \geq k + 1} x_j + \sum_{j \in S_{k - 1}}x_j >  x_k + \sum_{j \in S_{k - 1}}x_j \geq t,
    \]
    so we can find some $k' > k$ with $k' \notin S$. Therefore, there are arbitrarily large numbers that are not in $S$. Since $\lim_{j \to \infty}x_j = 0$, there is an $m \in \N \setminus S$ with $x_m < t - s$. Since $m \notin S$, we have $\sum_{j \in S_{m - 1}}x_j + x_m \geq t$ hence $\sum_{j \in S_{m - 1}}x_j \geq t - x_m > s$. Hence $S_{m - 1}$ is the desired finite subset.
\end{proof}

We shall apply the previous lemma with $x_j = -\log(1 - 1/p_j^2)$. The terms $x_j$ are positive and $\sum_{j \geq 1} x_j = \log{\zeta(2)}$ by \eqref{eq:1/zeta2}. However, the condition that $x_k < \sum_{j \geq k + 1} x_j$ fails to hold for all $k \in \N$. Here is how the numbers look like for the first few primes:
\begin{align*}
    x_1 &\approx 0.2877 & 
    x_2 &\approx 0.1178 & 
    x_3 &\approx 0.0408 & 
    x_4 &\approx 0.0206 \\
    \sum_{j \geq 2} x_j &\approx 0.2100 &
    \sum_{j \geq 3} x_j &\approx 0.0922 &
    \sum_{j \geq 4} x_j &\approx 0.0514 &
    \sum_{j \geq 5} x_j &\approx 0.0308 
\end{align*}
The inequality fails with $k = 1, 2$, and works with $k = 3, 4$. Let us verify that the condition does hold from that point on, so we can use the previous lemma with the sequence $(x_k)_{k \geq 3}$.

\begin{lemma}
    \label[lemma]{density condition}
    Let $p_j$ be the $j$-th prime and let $x_j = -\log(1 - 1 / p_j^2)$. Then $x_k < \sum_{j \geq k + 1} x_j$ for all $k \geq 5$.
\end{lemma}
\begin{proof}
    We first claim that $x_{j + 1} > x_j / 2$ for any $j \geq 5$. This is equivalent to 
    \begin{equation}
        \label{eq:equiv_condition}
        \left(1 - \frac{1}{p_{j + 1}^2}\right)^2 < 1 - \frac{1}{p_j^2}.
    \end{equation}
    It follows from \cite{axler2019new} that
    \[
    j\left(\log j + \log\log j - 3 / 2\right) < p_j < j\left(\log j + \log\log j - 1 / 2\right)
    \]
    for all $j \geq 20$. Basic analysis then gives $p_{j + 1} < 14 p_j / 10$
    for any $j \geq 28$, and hence
    \[
    \left(1 - \frac{1}{p_{j + 1}^2}\right)^2 = 1 - \frac{2}{p_{j + 1}^2} + \frac{1}{p_{j + 1}^4} < 1 - \frac{200}{144}\frac{1}{p_j^2} + \frac{1}{p_j^4} < 1 - \frac{1}{p_j^2}.
    \]
    It can be checked with a computer that \eqref{eq:equiv_condition} also holds for $j \in \{5, 6, \ldots, 27\}$. We obtain, for any $k \geq 5$,
    \[
    \sum_{j \geq k + 1} x_j > x_k \sum_{j \geq 1} \frac{1}{2^j} = x_k,
    \]
    as required.
\end{proof}

\begin{proposition}
    \label[proposition]{density}
    Let $(s, t)$ be a subinterval of
    \[
        \mathcal I = \left[ \frac{1}{\zeta(2)}, \frac23 \right] \cup \left[ \frac{9}{8 \zeta(2)}, \frac34 \right].
    \]
    Then there is a set of primes $p$ of positive lower density for which
    \[
        \prod_{q \mid p-1} \left( 1 - \frac{1}{q^2} \right) \in (s, t).
    \]
\end{proposition}
\begin{proof}
Let $x_j = -\log(1 - 1 / p_j^2)$. It follows by combining the previous two lemmas that the sums $\sum_{j \in S} x_j$ with $S$ a finite subset of $\N \setminus \{ 1, 2 \}$ form a dense subset of the interval $\left(0, \log\zeta(2) + \log(3 / 4) + \log(8 / 9)\right) = \left(0, \log (2\zeta(2)/3) \right)$. Equivalently, the products $\prod_{j \in S} e^{-x_j} = \prod_{j \in S} (1 - 1/p_j^2)$ form a dense subset of the interval $\left(3 / (2\zeta(2)), 1\right)$. 

Suppose first that $(s,t) \subseteq (1/\zeta(2), 2/3)$. By the argument above, we can find a set $S$ with $\prod_{j \in S}(1 - 1 / p_j^2) \in (3s/2, 3t/2)$. Let $S' = S \cup \{ 1, 2 \}$, so that $\prod_{j \in S'}(1 - 1 / p_j^2) \in (s, t)$. We can thus find a $k$ larger than all elements in $S'$ such that 
\[
\prod_{j \in S'}\left(1 - \frac{1}{p_j^2}\right) > s \cdot \exp \left(\sum_{j \geq k} x_j \right) = s \cdot \prod_{j \geq k} \left(1 - \frac{1}{p_j^2}\right)^{-1}.
\]
Take $X = \{ p_j \mid j \in S' \}$ and $Y = \{ p_j \mid j < k \} \setminus X$. Then any prime $p$ with the property that all primes from $X$ divide $p - 1$ and no prime from $Y$ divides $p - 1$ satisfies
\[
t > \prod_{q \mid p - 1}\left(1 - \frac{1}{q^2}\right) > \prod_{j \in S' \lor j \geq k}\left(1 - \frac{1}{p_j^2}\right) > s.
\]
The set of primes $p$ with this property has positive lower density by Lemma \ref{determined factors}.

The case when $(s,t) \subseteq (9/(8\zeta(2)),3/4)$ can be handled in a similar way by taking $S' = S \cup \{ 1 \}$. The interval $(3/(2\zeta(2)), 1)$ transforms under multiplication by $3/4$ to $(9/(8\zeta(2)), 3/4)$.
\end{proof}

\begin{corollary}
    Let $(s, t)$ be a subinterval of $\mathcal I$. Then there is a set of primes $p$ of positive lower density for which the following holds.
    Let $\alpha', \beta'$, $\gamma'$ be uniformly random in $\Z_n$, and let $\nu$ be the uniform distribution on $\Z_n$, where $n = p-1$. Then $\P_{\alpha', \beta', \gamma'}(\phi_k = \nu) \in (s, t)$ for all $k \geq 0$.
\end{corollary}

We now transport this result to the original setting of the random Lie bracket.

\begin{corollary}
    Let $(s, t)$ be a subinterval of $\mathcal I$. Then there is a set of primes $p$ of positive lower density with the following property.
    Let $A,B$ be uniformly random in $\slfrak_2(\F_p)$, let $\phi_k$ be the limiting distribution of the corresponding random walk on $\Z_n$ after $k$ steps\footnote{If $\alpha \beta \gamma = 0$, there is no corresponding walk on $\Z_n$. In that case, just take $\phi_k = 0$.}, let $\nu_{2k}$ be the corresponding distribution on $\lin \{ [A,B] \} \setminus \{ 0 \}$, and let $\nu$ be the uniform distribution on $\lin \{ [A,B] \} \setminus \{ 0 \}$. Then $\P_{A,B}(\nu_{2k} = \nu) \in (s, t)$ for all $k \geq 0$.
\end{corollary}
\begin{proof}
    Take any $A, B \in \slfrak_2(\F_p)$ such that $G_{A,B}$ has nonzero entries. Then the random Lie bracket satisfies $\nu_{2k} = \nu$ if and only if the parameters $(a, b) = F(A, B)$ satisfy $\gcd(n, a, b) = 1$. By \Cref{cor:pushforward_vs_uniform} we have $\abs{\P_{A,B}(\nu_{2k} = \nu) - \P_{a,b}(\gcd(n, a, b) = 1)} < 15 / p$ and by the previous corollary, we have
    \[
        \P_{a,b}(\gcd(n,a,b) = 1) = \nu_{\Z_n^3}\left( \{ (\alpha', \beta', \gamma') \in \Z_n^3 \mid \phi_k = \nu_{\Z_n} \} \right) \in (s + 15/p,t - 15/p)\footnote{We can assume $p$ is large enough so that $t - s > 30/p$.}
    \]
    for a set of primes $p$ of positive lower density. Combining these gives us that $\Pr_{A,B}(\nu_{2k} = \nu) \in (s,t)$ for a set of primes $p$ of positive lower density. 
\end{proof}

In particular, for any $x \in \mathcal I$ and any $\epsilon > 0$, there is a set of primes of positive lower density such that the random Lie bracket converges to the uniform distribution on a line after an even number of steps with probability in $(x - \epsilon, x + \epsilon)$.

\section{Pre-cutoff}
\label[section]{sec:precutoff}

Let $A,B \in \slfrak_2(\F_p)$ be uniformly random. Let $\mu_{2k}$ be the distribution of the random Lie bracket after $2k$ steps, and let $\nu_{2k}$ be the associated limiting distribution, supported on $\lin\{ [A,B] \}$. We shall now prove that with high probability, the distance $\dist(\mu_{2k}, \nu_{2k})$ quickly transitions from $1$ to $0$ at around $k \approx p$. This is the phenomenon of pre-cutoff, and we establish it by providing sharp upper and lower bounds on $\dist(\mu_{2k}, \nu_{2k})$ in terms of $k$ and the parameters $a,b$. We further inspect what happens for generic $a,b$.

The distribution of the random Lie bracket after $2k$ steps is the same as the distribution of the random walk on $\F_p^\ast$ with parameters $\alpha, \beta, \gamma$ after $k$ steps. Supposing these parameters are all nonzero, the random walk is equivalent to a random walk on $\Z_n$ with parameters $\alpha', \beta', \gamma'$. Let $\sigma_k$ be the distribution of this random walk after $k$ steps. The Markov operator has eigenvalues
\[
    \lambda_j = \frac{1}{2}\chi_j(-\alpha') + \frac{1}{4}\chi_j(-\beta') + \frac{1}{4}\chi_j(-\gamma')
\] 
of absolute value
\[
    \abs{\lambda_j} = \abs{\frac{1}{2} + \frac{1}{4}\chi_j(a) + \frac{1}{4}\chi_j(b)},
\]
where $a = \alpha' - \beta'$ and $b = \alpha' - \gamma'$.

\subsection{Upper bound}

For $x \in \R$, let $\langle x \rangle$ be the unique value in the interval $(-1/2, 1/2]$ such that $x - \langle x \rangle$ is an integer. Let us associate to each eigenvalue $\lambda_j$ the vector 
\[
u_j = \left(\left\langle \frac{ja}{n}\right\rangle, \left\langle \frac{jb}{n}\right\rangle\right) \in \R^2.
\] 
The norm $\norm{u_j}$ controls $\abs{\lambda_j}$ in the following way.

\begin{lemma}
    \label[lemma]{upper bound for eigenvalues}
    $\abs{\lambda_j} \leq \exp\left(-\frac{1}{2}\norm{u_j}^2\right)$.
\end{lemma}
\begin{proof}
    The triangle inequality gives
    \[
    \abs{\lambda_j} \leq 
    \frac14 \left( \abs{1 + \chi_j(a)} + \abs{1 + \chi_j(b)} \right) =
    \frac12 \abs{\cos\left(\pi ja/n\right)} + \frac12 \abs{\cos\left(\pi jb/n\right)}.
    \]
    Note that $\abs{\cos(\pi x)} = \abs{\cos(\pi\langle x \rangle)}$. Using the bound $\cos(\pi x) \leq \exp(-2x^2)$, which holds for all $x \in [-3/2, 3/2]$, we obtain
    \[
    \abs{\lambda_j} \leq \frac12 \left(
        \exp( -2\langle ja/n \rangle^2 ) + \exp( -2\langle jb/n \rangle^2 )
        \right).
    \]
    By Jensen's inequality for the function $x \mapsto \exp(-2x^2)$, which is concave on the interval $[-1/2, 1/2]$, the latter is at most
    \[
        \exp \left( -\frac{1}{2}\left(\abs{\langle ja / n\rangle} + \abs{\langle jb / n\rangle}\right)^2 \right) 
        \leq \exp \left(- \frac{1}{2} \norm{u_j}^2 \right)
    \]
    and the proof is complete.
\end{proof}

The characters contributing to the limiting distribution $\phi_k$ are those $\chi_j$ for which $\abs{\lambda_j} = 1$. Note that this happens if and only if $a,b \in \ker \chi_j$, which is the same as $u_j = 0$. In order to control the residual distribution $\sigma_k - \phi_k = \omega_k$, let $\Lambda$ be the plane lattice
\[
\Lambda = \bigcup_{j \in \Z_n} \left( u_j + \Z^2 \right) = 
\left( \frac{a}{n}, \frac{b}{n} \right) \Z + \Z^2
\subseteq \R^2
\]
and let $\minnorm$ be the minimal distance between two distinct points in $\Lambda$. When $a,b \neq 0$, the lattice $\Lambda$ is nontrivial, and we have
\[
\minnorm = \min \left\{ \norm{u_j} \mid j \in \Z_n, \ u_j \neq 0 \right\}.
\] 

\begin{proposition}
    \label[proposition]{prop:distance upper bound}
    For all $n$ and all $a, b \in \Z_n$, we have
    \[
    \dist(\sigma_k, \phi_k)^2 \leq 11 \gcd(a, b, n) e^{-k \minnorm^2}.
    \]
\end{proposition}
\begin{proof}
    Use \Cref{distance to limiting distribution} with the last lemma to bound\footnote{When $a,b = 0$, both sums are empty, and we have $\dist(\sigma_k, \phi_k) = 0$.} 
    \begin{equation}
        \label{eq:1}
        \norm{\sigma_k - \phi_k}_1^2 \leq 
        \sum_{|\lambda_j| < 1} \abs{\lambda_j}^{2k} \leq 
        \sum_{u_j \neq 0} \exp \left( -k\norm{u_j}^2 \right).
    \end{equation}
    Let us upper bound this sum in terms of the lattice $\Lambda$. Note that the union in the definition of $\Lambda$ might not be disjoint, since we might have $u_i = u_j$ for some $i,j$. This happens precisely when $n$ divides $a(j - i)$ and $b(j - i)$, which is equivalent to $n$ dividing $\gcd(a,b)(j - i)$, and this is the same as saying that $n/\gcd(a,b,n)$ divides $j-i$. Hence every value of $u_j$ occurs precisely $\gcd(a,b,n)$ times. We can thus bound the exponential sum in \eqref{eq:1} by
    \[
    \gcd(a,b,n) \sum_{u \in \Lambda \setminus \{ 0 \} } \exp \left( - k \norm{u}_2^2 \right).
    \]
    The sum of norms over the whole lattice $\Lambda$ can be upper bounded as in \cite[Lemma 6]{nocutoff} by
    \[
        9 \sum_{i \geq 1} (i + 1)^2 e^{-k \minnorm^2 i^2} \leq 
        18 \sum_{i \geq 1} (i + 1) e^{-k \minnorm^2 i}
        = 18 e^{-k \minnorm^2} \frac{2 - e^{- k \minnorm^2}}{(1 - e^{- k \minnorm^2})^2}.
    \]
    For $e^{-k \minnorm^2} \leq 1/10$, the value of the fraction is less than $12/5$, so we get the overall bound 
    \[
        \norm{\sigma_k - \phi_k}_1^2 \leq  44 \gcd(a,b,n) e^{-k \minnorm^2}. 
    \]
    For $e^{-k \minnorm^2} > 1/10$, the same bound holds since we always have $\norm{\sigma_k - \phi_k}_1^2 \leq 4$.
\end{proof}

The minimal norm $\minnorm$ can be estimated as follows.

\begin{lemma}
    \label[lemma]{lemma:minnorm_bound}
    We have
    \[
    \frac{\gcd(a,b,n)}{n} \leq \minnorm \leq \frac{2}{\sqrt{\pi}} \sqrt{\frac{\gcd(a, b, n)}{n}}.
    \]
\end{lemma}
\begin{proof}
    The lattice $\Lambda$ is contained in $(\gcd(a,b,n)/n)\Z^2$, hence $\Delta \geq \gcd(a,b,n) / n$. For the upper bound, project $\R^2$ to the torus $\R^2/\Z^2$. Observe that the image of the lattice $\Lambda$ contains exactly $n / \gcd(a, b, n)$ points. Since the minimal distance between two distinct points of $\Lambda$ is $\minnorm$, the open discs of radius $\minnorm / 2$ in the torus $\R^2/\Z^2$ around these points are disjoint. Comparing areas, we thus obtain
    \[
    \frac{n}{\gcd(a, b, n)} \pi \left(\frac{\minnorm}{2}\right)^2 \leq 1.
    \]
    Rearranging terms gives the claimed upper bound.
\end{proof}

The bounds in the lemma are sharp with respect to $n$. If $a = 1$ and $b = 0$, we have $\gcd(a, b, n) = 1$ and $\minnorm = 1/n$, matching the lower bound. For the upper bound, consider $n = m^2$, $a = m$, and $b = 1$ for some $m \in \N$. If $\norm{u_j} < 1/m$ for some $j \in \Z_n$, then, since the second coordinate of $u_j$ is less than $1/m$ in absolute value, we must have $\abs{j} < m$. However, this implies $j = 0$, as the same condition holds for the first coordinate. Thus $u_j = 0$. This demonstrates that $\minnorm \geq 1/m = 1/\sqrt{n}$, and the upper bound is also sharp up to a constant factor.

\subsection{Randomly chosen parameters}

Suppose the parameters $a,b \in \Z_n$ of the random walk are chosen uniformly at random. We shall now show that in this case, $\gcd(a,b,n)$ is not large and $\minnorm$ is of order $1 / \sqrt{n}$ with high probability.

\begin{lemma}
    \label[lemma]{lemma:small_gcd}
    Let $a, b$ be uniformly random in $\Z_n$. Then for any $M > 1$, we have
    \[
        \Pr_{a, b}(\gcd(a, b, n) \geq M) < 1/(M - 1).
    \]
\end{lemma}
\begin{proof}
    If $\gcd(a, b, n) \geq M$, then there is a $d \geq M$ that divides $a, b, n$. By the union bound, we thus get
    \[
        \Pr_{a, b}(\gcd(a, b, n) \geq M) \leq 
        \sum_{\substack{d |  n \\ d \geq M}} \Pr_{a,b}(a,b \in d \Z_n) < 
        \sum_{d \geq M} \frac{1}{d^2} < 
        \int_{M-1}^\infty \frac{1}{x^2} \, dx =
        \frac{1}{M-1}. \qedhere
    \]
\end{proof}

\begin{lemma}
    \label[lemma]{lemma:bound_gcd}
    Let $a, b \in \Z_n$ be uniformly random. For any $0 \neq j \in \Z_n$ and $r > 0$, we have
    \[
    \Pr_{a, b}\left(
        0 < \norm{u_j}  < r
        \right) < 8 r^2.
    \]
\end{lemma}
\begin{proof}
    Let $d = \gcd(n, j)$ and $m = n / d$. Let $L$ be the additive subgroup of $\R^2$ generated by $(d / n, 0)$ and $(0, d / n)$, and consider $\langle L \rangle = L \cap (-1/2, 1/2]^2$. For uniformly random $a,b$, the point $u_j$ is uniformly distributed on $\langle L \rangle$. Note that $\langle L \rangle$ consists of $m^2$ points and at most $(2\floor{rm} + 1)^2$ of them are of norm at most $r$. The probability that $0 < \norm{u_j} < r$ is thus bounded by $((2\floor{rm} + 1)^2 - 1)/m^2$. This is $0$ when $m < 1/r$, and for $m \geq 1/r$ it is at most $4r^2 + 4r / m \leq 8r^2$.
\end{proof}

\begin{lemma}
    \label[lemma]{lemma:upper bound for d}
    Let $a, b \in \Z_n$ be uniformly random. For any $\epsilon > 0$, we have
    \[
    \Pr_{a, b}\left(\minnorm < \epsilon / \sqrt{n}\right) < 8\epsilon^2.
    \]
\end{lemma}
\begin{proof}
    The event $\minnorm < \epsilon/\sqrt{n}$ implies that we must have $0 < \norm{u_j} < \epsilon/\sqrt{n}$ for some $0 \neq j \in \Z_n$. The claim now follows by the union bound and the previous lemma.
\end{proof}

\begin{theorem}
    \label[theorem]{thm:upper_bound}
    For every $\epsilon, \delta > 0$ there is a constant $C$ such that the following holds. Let $a, b \in \Z_n$ be uniformly random. Then for all $k > Cn$, we have
    \[
    \Pr_{a, b}\left( \dist(\sigma_k, \phi_k) < \epsilon \right) > 1 - \delta.
    \]
\end{theorem}
\begin{proof}
    With probability larger than $1 - \delta$ we have, by \Cref{lemma:small_gcd} and \Cref{lemma:upper bound for d}, that both $\gcd(a, b, n) < 1 + 2 / \delta < 3/\delta$ and $\minnorm^2 \geq \delta / 16n$ hold. Using these with \Cref{prop:distance upper bound} gives
    \[
    \dist(\sigma_k, \phi_k)^2 \leq \frac{33}{\delta} \exp \left( -\delta k/16n \right).
    \]
    For $k > Cn$, this is less than $\epsilon^2$, completing the proof.
\end{proof}

\begin{corollary}
    \label[corollary]{cor:upper_bound_tv}
    For every $\epsilon, \delta > 0$ there is a constant $C$ such that the following holds. Let $A, B \in \slfrak_2(\F_p)$ be uniformly random. Then for all $k > C p$, we have
    \[
    \Pr_{A, B}\left( \dist(\mu_{2k}, \nu_{2k}) < \epsilon \right) > 1 - \delta.
    \]
\end{corollary}
\begin{proof}
    Take $A, B \in \slfrak_2(\F_p)$ such that $G_{A,B}$ has nonzero entries. Then the random Lie bracket satisfies $\dist(\mu_{2k}, \nu_{2k}) < \epsilon$ if and only if the parameters $(a, b) = F(A, B)$ satisfy $\dist(\sigma_k, \phi_k) < \epsilon$. By the previous theorem, we have $\P_{a,b}(\dist(\sigma_k, \phi_k) < \epsilon) > 1 - \delta/2$ for $k > Cp$. It now follows from \Cref{cor:pushforward_vs_uniform} that $\Pr_{A,B}(\dist(\mu_{2k}, \nu_{2k}) < \epsilon) > 1 - \delta$.\footnote{We can assume $p$ is large enough so that $15/p < \delta/2$.}
\end{proof}

\subsection{Lower bound in \texorpdfstring{$L^1$}{L1}-distance}
We now prove the corresponding lower bound on the distance between $\sigma_k$ and $\nu_k$. Using Fourier analysis, we can easily obtain a lower bound $1 - \epsilon$ on the \emph{$L^1$-distance} (but not on the total variation distance, this will be done in the following section) in terms of the vectors $u_j$ and minimal norm $\minnorm$ as in the previous section. We tightly follow the argument in \cite{nocutoff}.

\begin{lemma}[Lemma 4 in \cite{nocutoff}]
    \label[lemma]{lemma:eigenvalue_lower_bound}
    If $\norm{u_j} \leq 1 / 2\pi$, then
    \[
        \abs{\lambda_j} \geq \exp\left(-\pi^2 \norm{u_j}^2\right).
    \]
\end{lemma}
\begin{proof}
    We have 
    \[
        \abs{\lambda_j} \geq 
        \real \left( \frac12 + \frac14 \chi_j(a) + \frac14 \chi_j(b) \right) = 
        \frac{1}{2} + \frac{1}{4}\cos\left(2\pi \langle ja/n\rangle\right) + \frac{1}{4}\cos\left(2\pi \langle jb/n\rangle\right).
    \]
    This can be lower bounded using $\cos x \geq \exp(-x^2)$, which holds for all $\abs{x} \leq 1$, by
    \[
        \frac{1}{2} + \frac{1}{4}\exp\left(-4\pi^2 \langle ja/n\rangle^2\right) + \frac{1}{4}\exp\left(-4\pi^2 \langle jb/n\rangle^2\right),
    \]
    and applying Jensen's inequality for the convex function $x \mapsto \exp(-x)$ gives 
    \[
        \abs{\lambda_j} \geq \exp\left(-\pi^2 \langle ja/n\rangle^2 - \pi^2 \langle jb/n\rangle^2\right) = \exp\left(-\pi^2 \norm{u_j}^2\right). \qedhere
    \]
\end{proof}

\begin{theorem}
    \label[theorem]{thm:lower_bound}
    For every $\epsilon, \delta > 0$ there are constants $C, c > 0$ such that the following holds. Let $n > C$ and let $a, b \in \Z_n$ be uniformly random. For all $k < cn$, we have
    \[
    \Pr_{a, b}\left(\dist(\sigma_k, \phi_k) > 1/2 - \epsilon \right) > 1 - \delta.
    \]
\end{theorem}
\begin{proof}
    With probability at least $1 - \delta$, we have $\gcd(a, b, n) \leq 1 + 1 / \delta$ by \Cref{lemma:small_gcd}. Let $C$ be such that
    \[
        2\sqrt{\frac{1 + 1 / \delta}{\pi C}} \leq \frac{1}{2 \pi}.
    \]
    As long as $n \geq C$, we thus have $\Delta \leq 1 / 2\pi$ by \Cref{lemma:minnorm_bound}. Let $u_j$ be the vector with $\Delta = \norm{u_j}$. We have $\abs{\lambda_j} \neq 1$ since $\norm{u_j} \neq 0$, and so \Cref{lemma:eigenvalue_lower_bound} gives $\rho \geq \abs{\lambda_j} \geq \exp(-\pi^2 \Delta^2)$. Hence
    \[
        \dist(\sigma_k, \phi_k) \geq \rho^k / 2 \geq \exp(-4\pi (1 + 1 / \delta)k / n) / 2
    \]
    by \Cref{distance to limiting distribution}. For $k < cn$, this is more than $1/2 - \epsilon$, completing the proof.
\end{proof}

\subsection{Lower bound in total variation distance}

Let us now prove the corresponding stronger lower bound of $1 - \epsilon$ in \emph{total variation distance}. We do this by refining the argument in \cite[Theorem 2]{hildebrand2005survey} (itself a modification of \cite[Theorem 6.1.1]{greenhalgh1989random}), which proves a bound of $1/2 - \epsilon$ in total variation distance (so as strong as the one in the previous section).

\begin{theorem}
    \label[theorem]{thm:lower_bound_tv}
    For every $\epsilon, \delta > 0$ there are constants $C, c > 0$ such that the following holds. Let $n > C$ and let $a, b \in \Z_n$ be uniformly random. For all $k < cn$, we have
    \[
    \Pr_{a, b}\left(\dist(\sigma_k, \phi_k) > 1 - \epsilon \right) > 1 - \delta.
    \]
\end{theorem}
\begin{proof}
    Let $\lambda > 0$ be a parameter to be determined later. Let $C, c > 0$ be such that
    \[
        \left(1 + 1 / \delta\right) \left( 2\lambda\sqrt{c} + 1/\sqrt{C} \right)^2
        < \epsilon.
    \]
    Assume $n > C$ and $k < cn$. We will exhibit a set $S \subseteq \Z_n$ that satisfies $\phi_k(S) < \epsilon$ and $\sigma_k(S) > 1 - 2\epsilon$, which in turn gives $\dist(\sigma_k, \phi_k) > 1 - 3\epsilon$. Let $\phi \colon \Z^3 \to \Z_n$ be the homomorphism defined by $\phi(1, 0, 0) = \alpha', \phi(0, 1, 0) = \beta', \phi(0, 0, 1) = \gamma'$, and let
    \[
        T = \left\{(k - x - y, x, y) \in \Z^3 \mid \abs{x - k/4}, \abs{y - k/4} \leq \lambda\sqrt{k}\right\}
        \quad \text{and} \quad
        S = \phi\left(T\right).
    \]
    The distribution $\phi_k$ is uniform on a subset of size $n / \gcd(a, b, n)$ by \Cref{lemma:distribution_nu}. With probability at least $1 - \delta$, we have $\gcd(a, b, n) \leq 1 + 1 / \delta$ by \Cref{lemma:small_gcd}, and hence
    \[
        \phi_k(S) \leq
        \abs{S}\gcd(a, b, n) / n \leq \left(1 + 1 / \delta\right)\left(2\lambda\sqrt{k} + 1\right)^2 / n 
        < \epsilon
    \]
    by the assumption on $n,k$. On the other hand, let $X, Y$ be the number of times the generators $\beta',\gamma'$ are chosen in the random walk. Note that $X,Y$ are distributed as a sum of $k$ independent Bernoulli random variables with success probability $1/4$. Thus
    \[
        \sigma_k(S) \geq 
        \Pr\left(\abs{X - k/4} \leq \lambda\sqrt{k}, \ \abs{Y - k/4} \leq \lambda\sqrt{k} \right)
        \geq 1 - 2 \Pr\left(\abs{X - k/4} > \lambda\sqrt{k}\right).
    \]
    By Hoeffding's inequality, the last probability is at most $2\exp\left(-2\lambda^2 \right)$. Let $\lambda$ be such that this is less than $\epsilon$. Then $\sigma_k(S) > 1 - 2\epsilon$, as required.
\end{proof}

\begin{corollary}
    \label[corollary]{cor:lower_bound_tv}
    For every $\epsilon, \delta > 0$ there are constants $C, c > 0$ such that the following holds. Let $p > C$ and let $A, B \in \slfrak_2(\F_p)$ be uniformly random. For all $k < cp$, we have
    \[
        \Pr_{A, B}\left( \dist(\mu_{2k}, \nu_{2k}) > 1 - \epsilon \right) > 1 - \delta.
    \]
\end{corollary}
\begin{proof}
    For matrices $A, B \in \slfrak_2(\F_p)$ such that $G_{A,B}$ has nonzero entries, the random Lie bracket satisfies $\dist(\mu_{2k}, \nu_{2k}) > 1 - \epsilon$ if and only if the parameters $(a, b) = F(A, B)$ satisfy $\dist(\sigma_k, \phi_k) > 1 - \epsilon$. By the previous theorem, we have $\P_{a,b}(\dist(\sigma_k, \phi_k) > 1 - \epsilon) > 1 - \delta/2$ for $k < cp$. We can assume $C$ is large enough so that $15/C < \delta/2$, so \Cref{cor:pushforward_vs_uniform} implies that $\Pr_{A, B}(\dist(\mu_{2k}, \nu_{2k}) < 1 - \epsilon) > 1 - \delta$.
\end{proof}

Corrollaries \ref{cor:upper_bound_tv} and \ref{cor:lower_bound_tv} together prove \Cref{thm:pre-cutoff}.

\section{Diameter}
\label[section]{diameter section}

Here, we prove the diameter bound of $O(\log p)$ for the random Lie bracket in $\slfrak_2(\F_p)$, as stated in \Cref{thm:diameter}. The key point behind the argument is based on the analysis of the affine random walk on $\F_p$ evolving according to $x_{n+1} = a x_n + b_n$ for fixed $a \in \F_p^\ast$ and i.i.d. random $b \in \F_p$. This walk has been studied many times in the literature, the starting point being \cite{chung1987random}. The most relevant for us are the strong results of \cite{breuillard2022cut} which exhibit rapid mixing of the walk for almost all values of $a$. In order to explain this more precisely, let, for any $a \in \F_p$ and $n \in \N$,
\[
    P_n(a) = \left\{\sum_{i=0}^{n} b_i a^i \mid b_i \in \{-1, 0, 1\} \right\} \subseteq \F_p
\]
be the set of all polynomials in $a$ of degree at most $n$ with coefficients in $\{-1, 0, 1\}$.

\begin{theorem}[Proposition 13 together with Theorem 2 in \cite{breuillard2022cut}]
    \label[theorem]{breuillard varju field diameter}
    There is a (possibly empty) small set of primes $T$ with the following property. For any $\epsilon > 0$ there is a $C$ such that for all primes $p > C$ not in $T$ and at least $(1 - \epsilon)p$ values of $a \in \F_p$, we have $P_n(a) = \F_p$ for some $n \leq 10\log p$.
\end{theorem}

\begin{remark}
Alternatively, we can use Theorem 1 from \cite{breuillard2022cut} instead of Theorem 2. This yields a weaker bound of $n \leq C_{\epsilon} \log p \log \log p$, but it holds for \emph{all} primes $p$.
\end{remark}

We now show how to do deduce the stated diameter bound from this. First of all, using basic properties of the random Lie bracket, we show that obtaining $P_n(a)$-multiples of $[A,B] \in \slfrak_2(\F_p)$ grows linearly with $n$.

\begin{lemma}
    \label[lemma]{lemma:line_diameter}
    Let $A, B \in \slfrak_2(\F_p)$ and $S = \{0, \pm A, \pm B\}$. Let $\alpha = 2 \langle A, B \rangle$. Then for any $n \in \N$, we have
    \[
        P_n(\alpha) \cdot [A, B] \subseteq S^{4n + 2}.
    \]
\end{lemma}
\begin{proof}
    We prove this by induction on $n$. The statement is trivial for $n = 0$. Suppose it holds for $n$ and take any $x \in P_{n + 1}(\alpha)$. Then $x = \alpha y + z$ for some $y \in P_n(\alpha)$ and $z \in \{-1, 0, 1\}$. By the induction hypothesis, we have $y[A, B] \in S^{4n + 2}$, and so $\alpha y[A, B] = [B, [A, y[A, B]]] \in S^{4n + 4}$ by \Cref{lemma:ad^2}. Then 
    \[
        x[A, B] = \alpha y[A, B] + z[A, B] \in S^{4n + 6} = S^{4(n + 1) + 2}.\qedhere
    \]
\end{proof}

Secondly, we show that, with high probability, the value $\alpha = 2 \langle A, B \rangle$ belongs to any large enough set.

\begin{lemma}
    \label[lemma]{lemma:good_alpha}
    Let $\mathcal A \subseteq \F_p$ be a subset of $\F_p$ with $\abs{\mathcal A} \geq (1 - \epsilon)p$ and let $A, B \in \F_p$ be uniformly random. Then
    \[
        \Pr_{A, B}\left(\alpha \notin \mathcal A\right) \leq 3\epsilon.
    \]
\end{lemma}
\begin{proof}
    Let $\mathcal X \subseteq \Sym_2(\F_p)$ be the set of symmetric matrices $X$ with $2 X_{12} \notin \mathcal A$. By assumption, we have $\abs{\mathcal X} \leq \epsilon p^3$. By \Cref{prop:gram_matrix_fibers}, we then get then bound $\abs{G^{-1}(\mathcal X)} \leq 3\epsilon p^6$. The result follows.
\end{proof}

We are now ready to deduce \Cref{thm:diameter}.

\begin{proof}[Proof of \Cref{thm:diameter}]
    Let $T$ be the exceptional set of primes from \Cref{breuillard varju field diameter}. Take any $\epsilon > 0$. Then there is a $C$ such that for all primes $p > C$ not in $T$ and at least $(1 - \epsilon)p$ values of $a \in \F_p$, we have $P_n(a) = \F_p$ for $n = \floor{10\log p}$.
    Hence, by \Cref{lemma:good_alpha}, for $A, B \in \slfrak_2(\F_p)$ uniformly random we have $P_n(\alpha) = \F_p$ with probability at least $1 - 3\epsilon$. This implies, by \Cref{lemma:line_diameter}, that $\lin\{[A, B]\} \subseteq S^{4n + 2}$. 
    
    We can assume that $C$ (and therefore $p$) is large enough that with probability at least $1 - \epsilon$, the matrices $A$, $B$ and $[A, B]$ span $\slfrak_2(\F_p)$ as a vector space.\footnote{Write $A = (a_{ij})$, $B = (b_{ij})$. Then $A,B,[A,B]$ span $\slfrak_2(\F_p)$ if and only if the $3 \times 4$ matrix of vectorized $A,B,[A,B]$ is of full rank, and so some minor is nonzero. By the Schwartz-Zippel lemma, this occurs with probability at least $1 - 4/p$.} In that case, we can write any element of $\slfrak_2(\F_p)$ as $x A + y B + z [A, B]$ for some $x, y, z \in \F_p$. The Gram matrix $G_{A,B}$ is invertible with probability at least $1 - 4/p > 1 - \epsilon$. Thus there are $r, s \in \F_p$ such that $(r, s) \cdot G_{A,B} = (y/4, -x/4)$, meaning that $-2\alpha r - 2\gamma s = x$ and $2\beta r + 2\alpha s = y$. Then, by \Cref{lemma:ad^2},
    \[
        x A + y B + z [A, B] = [A, r [A, B]] + [B, s [A, B]] + z [A, B] \in S^{12n + 8}.
    \]
    Hence $S^{12n + 8} = \slfrak_2(\F_p)$, and so with probability at least $1 - 5\epsilon$, we have the desired bound on the diameter.
\end{proof}

\bibliographystyle{alpha}
\bibliography{references}

\end{document}